\documentclass{article}
\usepackage{graphicx} 

\title{Categories meet semigroups in various ways}
\author{Bertalan Pécsi}
\date{2025}

\usepackage{amsthm}
\usepackage{bm}

\usepackage[matrix]{xy,xypic}

\usepackage{amssymb,amsmath}

\theoremstyle{definition}
\newtheorem{thm}{Theorem}[section]
\newtheorem{lemma}[thm]{Lemma}
\newtheorem{prop}[thm]{Proposition}

\newtheorem{deff}[thm]{Definition} 
\newtheorem{ex}[thm]{Examples}
\newtheorem{rem}[thm]{Remark}

\newcommand{\ax}[1]{\pmb{\textsc{#1.}}}
\newcommand{\emp}[1]{\textsl{#1}}
\newcommand{\ct}{\mathcal}

\newcommand{\dia}{\xymatrix}
\newcommand{\al}{\underset}
\newcommand{\fl}{\overset}
\newcommand{\Ob}{\mathop{\rm Ob}}
\newcommand{\fii}{\varphi}
\newcommand{\eps}{\varepsilon}
\newcommand{\ro}{\varrho}
\newcommand{\Th}{\Theta}
\renewcommand{\c}{\cdot}
\renewcommand{\b}{\fl{ }{\bullet}} 
\newcommand{\s}[1]{{}^{\ast}\!{#1}}
\renewcommand{\S}{\mathrm S}
\newcommand{\K}{\mathrm K}
\newcommand{\Set}{\ct Set}
\newcommand{\tild}{\widetilde}
\newcommand{\barr}{\overline}
\renewcommand{\P}{\mathbf P}

\newcommand{\xxi}{{\bm\xi}}

\begin{document}

\maketitle

\begin{abstract}
This paper touches on several interaction points of semigroups and constructions from category theory.   
\end{abstract}

\section*{Introduction}
(...)

Throughout the paper we write compositions of arrows of a category in the order of application left to right, 
and consequently apply functions and functors to the \emp{right} to their arguments, often in exponents,
so that $x^f$ (or $xf$) denotes the image of element $x$ under mapping $f$.

Accordingly, epimorphisms, monomorphisms, split epi, (...)

\section{Wired categories}

\begin{deff}
We call a category $\ct C$ equipped with a chosen arrow $w_{a,b}:a\to b$ for each pair of 
objects $a,b\in\Ob\ct C$ a \emp{wired category} if $w_{a,a}=1_a$ for each $a\in \Ob\ct C$.
\end{deff}

\begin{prop} \label{wired-semigroup}
Any small wired category $(\ct C,w)$ induces a semigroup structure on its set of arrows,
the operation of which extends the original composition of $\ct C$.
\end{prop}
\begin{proof}
For arrows $\fii:a\to b$ and $\psi:c\to d$ define
$$ \fii \,\b\, \psi \ := \ \fii \c w_{b,c} \c  \psi \,.$$
This is an associative operation, since if $\ro:e\to f$ is also arbitrary,
we get $\fii\b\psi\b\ro\ = \fii \c w_{b,c}\c \psi\c w_{d,e}\c \ro$.\\
For consecutive arrows, i.e. if $b=c$, $\fii\b\psi\ =\ \fii\c\psi$ because $w_{b,b}$
was assumed to be the identity.
\end{proof}

\noindent
Recall that the \emp{Karoubi envelope} (alternatively, \emp{idempotent} or \emp{Cauchy
completion}) of a semigroup $S$ consists of the idempotents of $S$ as objects
and triples $(e,x,f)\in S^3$ as arrows $e\to f$ where $e,f$ are idempotents and $exf=x$.
The triple $(e,e,e)$ serves as the identity on $e$, and the composition is defined as follows:
$$(e,x,f) \,\c\, (f,y,g)\ := \ (e,\,xy,\,g)\,.$$

\begin{deff}
Let $(\ct C,w)$ and $(\ct D,v)$ be wired categories. A functor $F:\ct C\to\ct D$ is called
a \emp{wired functor} if for all objects $a,b\in\Ob\ct C$ we have
$\left( w_{a,b} \right)^F\; = \; v_{a^F\!,\,b^F} $\,. \\
We denote by $\ct{WC}at$ the category of small wired categories and wired functors. 
\end{deff}
\noindent
Note that there would be no further restriction on natural transformations
between wired functors.

\begin{prop}
The Karoubi envelope functor $\K:\ct Sgr \to \ct Cat$ factors through $\ct{WC}at$, and
it is right adjoint to the construction $\S:\ct{WC}at\to\ct Sgr$ in \ref{wired-semigroup}.
\end{prop}
\begin{proof}
Define the wires of the Karoubi envelope by $w_{e,f}\, := \, ef$ for idempotents $e,f\,\in S$.\\
The elements of both homsets in the adjunction to be proved, naturally correspond to maps 
from the set of arrows of a wired category $(\ct C,w)$
to the semigroup $S$ that converts composition in $\ct C$ to the operation of $S$.
\end{proof}

\begin{prop}
A semigroup $S$ has enough idempotents (for every $x$ there are idempotents
$e,f$ such that $exf=x$) if and only if $S^{\K\S}\to S$, the counit of the 
above adjunction evaluated at $S$, is surjective.
\end{prop}

\begin{deff}
The \emp{principal identifier relation} of a small wired category $(\ct C,w)$ is the
binary relation $\Th$ on the set of arrows of $\ct C$, for which 
$\al{a\to x}\alpha \,\Th\, \al{b\to y}\beta$
holds if and only if the following diagram commutes.
$$ \dia{a \ar[r]^{w_{a,b}} \ar[d]_\alpha \ar[rrd]_(.3)\alpha & 
b \ar[r]^{w_{b,a}} \ar[d]^(.3){\!\beta} \ar[rrd]_(.3)\beta & 
a \ar[r]^{w_{a,b}} \ar[d]^(.3){\!\alpha} & b \ar[d]^\beta \\ 
x \ar[r]_{w_{x,y}} & y \ar[r]_{w_{y,x}} & x \ar[r]_{w_{x,y}} & y } $$
\end{deff}
It is clear that $\Th$ is reflexive and symmetric, but in general it is not transitive.
Also, if $\alpha,\beta$ are parallel arrows $a\to x$ (i.e. $a=b,\ x=y$ above) then
$\alpha \,\Th\, \beta$ implies $\alpha=\beta$.

(....)

\section{Regular semigroups}
Recall that a semigroup is called \emp{regular} if every element $x$ has a \emp{pseudoinverse}
$y$, in the sense that $xyx=x$.

\noindent
Also recall that an arrow in a category is called a \emp{split epimorphism} if it has a left inverse,
and dually \emp{split monomorphisms} are the arrows which have right inverse. 

\begin{prop} \label{splitfact-regular}
Let $(\ct C,w)$ be a wired category such that each arrow $\fii\in \ct C$ can be written as $\fii=\eps\c\mu$
where $\eps$ is split epimorphism and $\mu$ is split monomorphism. 
Then the semigroup $S=(\ct C,w)^\S$ is regular.
\end{prop}
\begin{proof}
Let $x\in S$, i.e. $x\in \ct C$ be arbitrary with $x:a\to b$, and write it as $x=em$ with
$e:a\to c$ and $m:c\to b$ according to the hypothesis.
Let $e':c\to a$ be a left inverse of $e$ and $m':b\to c$ be a right inverse of $m$.
Set $y:= m'\,e'$. Then we have
$$xyx\ =\ em\; m'e' \;em\ =\ e\,1_c\,1_c\,m\ = \ em\ =\ x \,.$$
\end{proof}

\begin{lemma}
Assume that every arrow $\fii$ in a category $\ct C$ can be decomposed as $\fii = \eps\c\mu$
where $\eps$ is split epimorphism and $\mu$ is split monomorphism. 
Then $\ct C$ admits a split epi -- split mono factorization system.
\end{lemma}
(...)

\begin{thm}
A semigroup $S$ is regular if and only if it has enough idempotents and
its Karoubi envelope admits a \emp{split epi} -- \emp{split mono} factorization.
\end{thm}
\begin{proof}
First, if $S$ is regular, then any $x\in S$ is supported by idempotents $xy$ and $yx$, 
where $y$ is a pseudoinverse of $x$. \\
Let $(e,x,f)$ be an arrow in the Karoubi envelope $S^\K$, i.e. $e,f$ are idempotents
and $ex=x$ and $xf=x$, and let $y$ be a pseudoinverse of $x$. Then by $x=xyx$ we have
$$(e,x,f) = (e,x,yx) \,\c \, (yx, yx, f) \,. $$
Here $(yx,yxye,e)$ is a left inverse of $(e,x,yx)$ because 
$$(yx,yxye,e) \,\c\, (e,x,yx)\ =\ (yx,\, yxy\,ex,\, yx)\ =\ (yx,yx,yx)\ =\ 1_{yx} \,,$$
and $(f, fyx, yx)$ is a right inverse of $(yx,yx,f)$ because
$$(yx,yx,f) \,\c\, (f,fyx,yx)\ =\ (yx,\, yx\,fyx,\, yx)\ =\ (yx,yx,yx)\ =\ 1_{yx}\,.$$

\noindent
The converse direction follows from prop.\ref{splitfact-regular}.
\end{proof}

\section{Lax spiders}
Another approach in studying the category-like structure determined by a semigroup $S$ is that instead of
a single object we enable a whole set of objects as the ``domain'' / ``codomain'' of an element of $S$. \\
Here we describe one way to formulate this point of view. \\

\begin{deff}
Let $S$ be a semigroup and let $T$ be an ordered semigroup. We call a function $f:S\to T$
\emp{lax morphism} or \emp{subhomomorphism} if for all $x,y\in S$ it satisfies
$$ x^f \c y^f \le (x\c y)^f \,. $$
\end{deff}
Note that lax morphisms in the above sense correspond to \emp{lax functors} in a suitable setup 
when both $S$ and $T$ are turned into monoidal categories
(by adding an identity element and other necessary structure).\\

Let $E$ denote the set of idempotents in $S$, $\P(E)$ the set of its subsets, and
consider the rectangular band on $\P(E)\times \P(E)$ (so that its operation is
$(A,B)\,\c\,(C,D)\ =\ (A,D)$) with partial order $(A,B) \le (A',B')$ 
whenever $A\subseteq A'$ and $B\subseteq B'$. \\

\noindent
Then we get a lax morphism $S\to \P(E)\times \P(E)$ by assigning
$$ x \ \mapsto \ \left(\{e\in E: ex=x\},\; \{e\in E: xe=x\} \right) \,. $$

\section{Semigroupads}
\begin{deff}
A \emp{semigroupad} or \emp{semimonad} on a category $\ct C$ is a semigroup object in the category of 
endofunctors on $\ct C$, i.e. a monad $T:\ct C \to \ct C$ without the requirement of the unit 
transformation $\eta:1_{\ct C} \to T$ and its defining identities.\\
So, a semigroupad on $\ct C$ is a functor $T:\ct C \to \ct C$ equipped with a natural 
transformation $\mu:TT\to T$ which is 'associative' in the sense that $T\mu \c \mu = \mu T \c \mu$, i.e. 
for each object $a\in \Ob\ct C$ the following diagram commutes:
$$\dia{a^{TTT} \ar[r]^{a^{T\mu}} \ar[d]_{a^{\mu T}} & a^{TT} \ar[d]^{a^\mu} \\ 
a^{TT} \ar[r]_{a^\mu} & a^T} $$
\end{deff}

\begin{ex}[on $Set$] \phantom{.}
\begin{itemize}
  \item Let $T$ be the functor $A\mapsto A\times S$ for any fixed semigroup $S$, 
and let $\mu$ be the free $S$-action, i.e., 
$$A^\mu\ =\ \left((a,s_1),\,s_2 \right)\mapsto (a, \,s_1s_2)\, : (A\times S)\times S\to A\times S \,.$$
  \item (Finite) lists or subsets of size greater than a given minimum.
  \item Nonprincipal ultrafilters.
\end{itemize}
\end{ex}

\begin{thm}
For a given functor $T:\ct C\to\ct C$, a semigroupad structure $\mu:TT\to T$ determines and is
determined by a \emp{Kleisli lifting} operation which assigns to each arrow $f$ of the form 
$a\to b^T$ an arrow $\s f:a^T\to b^T$ so that the following conditions hold:
\begin{enumerate}
    \item[\ax{k1}] If $f:a\to b$ and $g:b\to c^T$ then $\s{(f\,g)} = f^T\, \s g$
    \item[\ax{k2}] If $f:a\to b^T$ and $g:b\to c$ then $\s{(f\,g^T)} = \s f\, g^T$
    \item[\ax{k3}] If $f:a\to b^T$ and $g:b\to c^T$ then $\s{(f\,\s g)} = \s f\; \s g$
\end{enumerate}
\end{thm}
\begin{proof}
Assume first that the natural transformation $\mu:TT\to T$ defines a semigroupad structure, 
that is, $T\mu \c \mu = \mu T \c \mu$. 
The same way as for the Kleisli lifting for monads, for an arrow $f:a\to b^T$ we can define
$$\s f \ := \ f^T \c b^\mu\,. $$
\ax{k1} \ If $f:a\to b$ and $g:b\to c^T$ then 
$ \s{(f\,g})\ = \ f^T \, g^T\, c^\mu \ = \ f^T\, \s g$. \\[6pt]
\ax{k2} \ If $f:a\to b^T$ and $g:b\to c$ then, by naturality of $\mu$,
$$ \s{(f\,g^T)} \ = \ f^T\,g^{TT}\,c^\mu \ = \ f^T\,b^\mu\,g^T \ = \ \s f\,g^T \, . $$
\ax{k3} \ If $f:a\to b^T$ and $g:b\to c^T$ then, substituting $\s g=g^T\,c^\mu$ and applying
$\mu T\c \mu = T\mu \c \mu$ and naturality of $\mu$, we obtain
$$ \s{(f\,\s g)} \ =\ f^T\, g^{TT} c^{\mu T}\, c^\mu \ = \ 
f^T\, g^{TT} \, c^{T\mu}\, c^\mu \ = \ f^T\, b^\mu \, g^T\, c^\mu \ = \ \s f\; \s g \, . $$
For the converse direction, assume a Kleisli lifting is given. For any object $a\in\Ob\ct C$ define
$$a^\mu\ :=\ \s{(1_{a^T})} \, : a^{TT} \to a^T \,.$$
This is a natural transformation $TT\to T$ because for any arrow $f:a\to b$ in $\ct C$ we have
$$ f^{TT}\, b^\mu \ = \ f^{TT}\, \s{(1_{b^T})} \ \fl{\ax{k1}}= \ \s{(f^T\,1_{b^T})} \ = \ 
\s{(1_{a^T}\,f^T)} \ \fl{\ax{k2}}= \ \s{(1_{a^T})}\,f^T \ = \ a^\mu\,f^T \,. $$
Associativity of $\mu$ follows by
$$ a^{T\mu}\,a^\mu \ = \ \s{(1_{a^{TT}})} \, \s{(1_{a^T})} \ \fl{\ax{k3}}= \ 
\s{(1_{a^{TT}}\, \s{(1_{a^T})})} 
\ = \ \s{(a^\mu)}$$
$$ a^{\mu T}\, a^\mu \ = \ a^{\mu T}\, \s{(1_{a^T})} \ \fl{\ax{k1}}= \ 
\s{(a^\mu\,1_{a^T})} \ = \ \s{(a^\mu)} \,. $$
Finally, it is also straightforward to verify that these procedures are inverses to each other.
\end{proof}

\begin{rem}
As we can read from this proof, actually there is a one-to-one correspondence between
all natural transformations $\mu:TT\to T$ and Kleisli $\ast$ operations satisfying \ax{k1} and \ax{k2}
\end{rem}

\begin{rem}
The Kleisli category construction for a monad still gives rise to a \emp{semicategory} structure
(i.e., a category without the requirement of identity arrows), when the monad is replaced by a semigroupad.
The proof relies on \ax{k3}
\end{rem}

The following theorem, inspired by (...), is a generalization of the celebrated property of the 
ultrafilter monad $\beta:\Set \to \Set$  that if $S$ is a semigroup, then $S^\beta$ also admits 
a semigroup structure, extending the original one. It turns out that this property extends to
\emp{all} monads on $\Set$ and even to semigroupads.
\begin{thm}  \label{meta-sgr}
Let $T$ be a semigroupad on $\Set$ and let $S$ be a semigroup, then there is an 
induced semigroup operation on $S^T$.
\end{thm}
\noindent
To achieve this, we introduce the following notation.\\
For an element $y\in S$, let $(\c y)$ denote the map $S\to S$ sending $x\mapsto xy$, and 
for any $p\in S^T$ we define
$$\bar p: S \to S^T,\ \ x\mapsto p\,(\c x)^T $$
so that $x\,\bar p$, the image of $x\in S$ under $\bar p$, is defined to be the image of $p\in S^T$ 
under the map $(\c x)^T:S^T\to S^T$.
Finally, set 
$$\tild p\ :=\ \s{\bar p}\ :\, S^T\to S^T $$
and define the operation for $p,q\in S^T$ as
$$p\b q := q\,\tild p\, .$$

\begin{lemma} Under the hypothesis of the theorem, for any $y\in S$ and $p\in S^T$ we have
$$(\c y)\,\bar p \ = \ \bar p\,(\c y)^T $$
$$(\c y)^T\,\tild p\ =\ \tild p\,(\c y)^T$$
\end{lemma}
\begin{proof}
Applying both sides of the first equation to any element $x\in S$ and using 
associativity of the operation of $S$, we get
$$x\,(\c y)\,\bar p\ = \ (xy)\,\bar p\ =\ p\,(\cdot(xy))^T\ =\ p\,(\c x)^T\,(\c y)^T \ = \  
x\,\bar p\,(\c y)^T\,.$$
Then the second equation is obtained by Kleisli lifting the first one:
$$(\c y)^T\,\tild p\ = \ (\c y)^T\,\s{\bar p}\ \fl{\ax{k1}}=\ \s{\left( (\c y)\,\bar p\right)}\ =\ 
\s{\left(\bar p\,(\c y)^T \right)} \  \fl{\ax{k2}}= \ \s{\bar p}\,(\c y)^T\ = \ \tild p\,(\c y)^T \,.$$
\end{proof}

\begin{lemma} \label{tilde-dot}
Under the hypothesis of the theorem, for any $p,q \in S^T$ we have
$$\tild{p\b q} = \tild q\,\tild p$$
\end{lemma}
\begin{proof}
Using the previous lemma, for any $y\in S$ we have 
$$y\left(\barr{p\b q} \right) \ =\ (p\b q)\,(\c y)^T\ =\ q\,\tild p\,(\c y)^T \ = \ 
q\,(\c y)^T\,\tild p\ =\ y\,\bar q\,\tild p $$
yielding $\barr{p\b q}\ =\ \bar q\,\tild p$, so by \ax{k3} we arrive to $\tild{p\b q}\ =\ \tild q\,\tild p$.
\end{proof}

\begin{proof}[Proof of theorem \ref{meta-sgr}]
We show that the operation $\b$ on $S^T$ is associative. For any $p,q,r\in S^T$,
by lemma \ref{tilde-dot} we have
$$ (p\b q)\b r \ = \ r \left(\tild{p\b q}\right) \ = \ r\,\tild q\,\tild p \ 
= \ (q\b r)\,\tild p \ = \ p \b (q\b r) \, .$$
\end{proof}

\begin{deff} Let $(T,\mu)$ be a semigroupad on a category $\ct C$.
A natural transformation $\xi:1_{\ct C} \to T$ is called a
\begin{itemize}
  \item \emp{left identity} \ if \ $\xi T\c\mu = 1_T$,
  \item \emp{right identity} \ if \ $T\xi \c \mu = 1_T$,
  \item \emp{central constant} \ if \ $\xi T\c\mu \,=\, T\xi\c\mu$,
  \item \emp{idempotent constant} \ if \ $\xi \c \xi T\c\mu \,=\, \xi$.
\end{itemize}
\end{deff}
\noindent
We note that any natural transformation $\xi:1_{\ct C} \to T$ satisfies 
$\xi \c \xi T \ = \ \xi \c T\xi$, so $\xi$ is an idempotent constant if and only if
$\xi \c T\xi \c \mu\ = \ \xi$.

\begin{prop}
Let $(T,\mu)$ be a semigroupad on a category $\ct C$ with Kleisli lifting $f\mapsto \s f$,
and  let $\xi:1_{\ct C}\to T$ be a natural transformation.
Then the following statements hold:
\begin{enumerate}
  \item $\xi$ is a left identity \ if and only if \ for each $a\in\Ob\ct C$ we have
$$\s{(a^\xi)} \ = \ 1_{a^T} \,. $$
  \item $\xi$ is a right identity \ if and only if \ for each $f:a\to b^T$ we have
$$ a^\xi \c \s f\ = \ f  \,.$$
  \item $\xi$ is a central constant \ if and only if \ for each $f:a\to b^T$ we have
$$ a^\xi \c \s f \ = \ f \c \s{(b^\xi)} \,. $$
  \item $\xi$ is an idempotent constant \ if and only if \ for each $a\in\Ob\ct C$ we have
$$ a^\xi \c \s{(a^\xi)} \ = \ a^\xi \,. $$
\end{enumerate}
\end{prop}
\begin{proof} Recall that for an arrow $f:a\to b^T$ we have $\s f \, = \, f^T\,b^\mu$.
\begin{enumerate}
  \item Obvious.
  \item If $\xi$ is a right identity, then by naturality of $\xi$ we get
$$ a^\xi \c \s f \ =\ a^\xi \c f^T \c b^\mu \ = \ f \c b^{T\xi} \c b^\mu \ = \ f \,.$$
Conversely, apply the hypothesis to $f=1_{a^T}:a^T\to a^T$ to get 
$$ 1_{a^T} \ = \ a^{T\xi} \c \s{(1_{a^T})} \ = \ a^{T\xi} \c a^\mu \,. $$
  \item If $\xi$ is central, then similarly as in 2., we have
$$ a^\xi \c \s f \ = \ f \c b^{T\xi} \c b^\mu \ = \ f\c b^{\xi T} \c b^\mu \ = \ 
f \c \s{(b^\xi)} \,. $$
Conversely, applying again the hypothesis to $f=1_{a^T}$, we arrive at
$$ a^{T\xi} \c a^\mu \ = \ a^{T\xi} \c \s{(1_{a^T})} \ = \ 1_{a^T} \c \s{(a^\xi)} 
\ = \ a^{\xi T} \c a^\mu \,.$$
  \item This readily follows from
$$a^\xi \c \s{(a^\xi)} \ = \ a^\xi \c a^{\xi T} \c a^\mu \,.$$
\end{enumerate}
\end{proof}

\begin{thm}
Assume a semigroupad $T$ on $\Set$ is given with Kleisli lifting $f\mapsto \s f$,
let $\xi:1_{\Set}\to T$ be a natural transformation, and let $(S,\c)$ be a semigroup
with a distinguished element $e\in S$. 
As a shorthand, we write $\xxi$ for $S^\xi:S\to S^T$ 
below.\\
Then the following statements hold:
\begin{enumerate}    
  \item If $\xi$ is a left identity for $T$ and $e$ is a left identity in $S$,
then $e^\xxi$ is a left identity in the induced semigroup on $S^T$.
  \item If $\xi$ is a right identity for $T$ and $e$ is a right identity in $S$,
then $e^\xxi$ is a right identity in the induced semigroup on $S^T$.
  \item If $\xi$ is a central constant for $T$ and $e$ is a central element in $S$,
then $e^\xxi$ is a central element in the induced semigroup on $S^T$.
  \item If $\xi$ is an idempotent constant for $T$ and $e$ is an idempotent element in $S$,
then $e^\xxi$ is an idempotent element in the induced semigroup on $S^T$.
\end{enumerate}
\end{thm}
\begin{proof}
Let $p\in S^T$ be an arbitrary element. Note that by naturality we have
$f\,\xxi \,=\, \xxi\,f^T $ for any $f:S\to S$.
\begin{enumerate}
  \item $e^\xxi \b p \ = \ p\;\tild{e^\xxi}\ =\ p\;\s{\left(x\mapsto e^\xxi \,(\c x)^T \right)} \ 
= \ p\;\s{\left(x\mapsto e \, (\c x) \, \xxi \right)}\ =\ p\;\s\xxi\ =\ p$.
  \item $p \b e^\xxi \ = \ e^\xxi\;\tild p \ = \ e^\xxi\;\s{\left(x\mapsto p\,(\c x)^T \right)} \
= \ e \left(x\mapsto p\,(\c x)^T \right) \ = \ p\,(\c e)^T\ = \ p$.
  \item $e^\xxi \b p \ = \  p\;\s{\left(x\mapsto (ex)^\xxi \right)}\ =\ 
p\;\s{\left(x\mapsto (xe)^\xxi \right)}\ = \ p\,\s{\left((\c e)\,\xxi \right)}$, \\
  $p \b e^\xxi \,= \, e^\xxi\;\s{\left(x\mapsto p\,(\c x)^T \right)} \, = \, 
e \left( x\mapsto p\,(\c x)^T \right) \s\xxi  \, = \, p\,(\c e)^T\;\s\xxi\, = \, 
p\,\s{\left((\c e)\,\xxi \right)}$.
  \item $e^\xxi \b e^\xxi \, =\, e^\xxi\;\s{\left(x \mapsto e^\xxi(\c x)^T \right)}\, = \, 
e^\xxi\;\s{\left(x \mapsto (ex)^\xxi \right)} \, = \, 
e^\xxi \; \s{\left( (e\c)\,\xxi \right)}\, =$ \\ 
  $= \, e^\xxi\; \s{\left( \xxi  \,(e\c)^T\right)} \ = e\,\xxi\, \s\xxi\,(e\c)^T \ =\ 
e\,\xxi \, (e\c)^T\ =\ e (e\c)\, \xxi\ =\ e^\xxi $.
\end{enumerate}

\end{proof}


\begin{thebibliography}{99}
\bibitem{CS}[A.~Costa, B.~Steinberg, 2014] The Schützenberger category of a semigroup, 
{\it arXiv:1408.1615}

\bibitem{H}[S.~Hayashi, 1985] Adjunction of semifunctors: categorical structures in nonextensional lambda calculus,
{\it Theoretical Computer Science 41}


\end{thebibliography}
\end{document}